\documentclass[]{article}
\usepackage{graphicx}
\usepackage{amsfonts}
\usepackage{amsmath}
\usepackage{amssymb}
\usepackage{fancyhdr}
\usepackage{titlesec}
\usepackage{indentfirst}
\usepackage{booktabs}
\usepackage{verbatim}
\usepackage{color}
\usepackage{amsthm}
\usepackage[page,header]{appendix}
\usepackage{titletoc}

\newcommand{\rd}{\,\mathrm{d}}

\newcommand{\opF}{\mathcal{F}}
\newcommand{\opC}{\mathcal{C}}

\numberwithin{equation}{section}
\newtheorem{theorem}{Theorem}[section]
\newtheorem{lemma}[theorem]{Lemma}
\newtheorem{corollary}[theorem]{Corollary}

\newtheorem{remark}[theorem]{Remark}

\begin{document}

\title{A study of Landau damping with random initial inputs
\footnote{This work was partially supported by NSF grants DMS-1522184 and DMS-1107291: RNMS KI-Net, and by
		the Office of the Vice Chancellor for Research and Graduate Education at the University of Wisconsin-Madison
		with funding from the Wisconsin Alumni Research Foundation.}}
\author{Ruiwen Shu\footnote{
		Department of Mathematics, University of Wisconsin-Madison, Madison, WI 53706, USA (rshu2@wisc.edu).}, Shi Jin\footnote{
		Department of Mathematics, University of Wisconsin-Madison, Madison, WI 53706, USA (sjin@wisc.edu).}}
\maketitle

\begin{abstract}
For the Vlasov-Poisson equation with random uncertain initial data, we prove that the Landau damping solution given by the deterministic counterpart (Caglioti and Maffei, {\it J. Stat. Phys.}, 92:301-323, 1998) depends smoothly on the random variable if the time asymptotic profile does, under the smoothness and smallness assumptions similar to the deterministic case. The main idea is to generalize the deterministic contraction argument to more complicated function spaces to estimate derivatives in space, velocity and random variables. This result suggests that the random space regularity can persist in long-time even in time-reversible nonlinear kinetic equations.
\end{abstract}

\section{Introduction}

In this paper we are concerned with the Vlasov-Poisson (VP) equation, which is a widely used model in plasma physics~\cite{LP, KT}. The VP equation reads
\begin{equation}\label{VP0}
\partial_t f + v\cdot\nabla_x f + E\cdot\nabla_v f = 0,\quad  E = \nabla_x \Delta_x^{-1}(\rho-\rho_0),\quad \rho = \int f \rd{v},
\end{equation}
with initial data $f(x,v,0)=f_{in}(x,v)$, where $t\in \mathbb{R}^+$ is the time variable, $x\in \mathbb{T}^d=[0,2\pi]^d$ is the space variable, and $v\in \mathbb{R}^d$ is the velocity variable. $f=f(x,v,t)$ is the particle distribution function of electrons. The term $E\cdot\nabla_v f$ represents the effect of the self-consistent electric field $E=E(x,t)$ on the electrons. $\rho=\rho(x,t)$ is the local density of electrons, while the constant $\rho_0 =\int\int f_{in}\rd{v}\rd{x}$ is the background charge (from ions) satisfying the neutrality condition.

Landau damping, first discovered by Landau~\cite{Lan} in 1946 in the linearized setting, is one of the most famous physical phenomena for the VP equation. It says that given the initial data close enough to some spatial homogeneous equilibrium, i.e., $f_{in}(x,v) = f^0(v) + h(x,v)$ where $h$ is small, then the electric field decays exponentially in time, if $f^0(v)$ satisfies certain conditions. Since the discovery of Landau damping, there has been a few work at the linearized level~\cite{Deg,MF,GS1,GS2}, but the first nonlinear result was obtained by Caglioti and Maffei~\cite{CM} in 1998 (and later improved by Hwang and Vel\'azquez~\cite{HV}). Using the scattering approach, Caglioti-Maffei proved that there exists a class of analytic initial data such that Landau damping does happen, by a fixed-point argument under the perturbative setting. Then in 2011 Mouhot and Villani~\cite{MV} proved Landau damping for all small initial data. In the same year Lin and Zeng~\cite{LZ1,LZ2} proved that Landau damping is not true for Sobolev initial data with low regularity index. After that Bedrossian, Masmoudi and Mouhot~\cite{BMM} generalized this result to solutions with Gevrey regularity. We refer to~\cite{MV} for a thorough review of the history of Landau damping, and~\cite{BMM} for recent progress.

All previous results are concerned with deterministic initial data, but in reality the initial data come from experiments, thus may have uncertainty due to measurement error. It is important to quantify the uncertainty, i.e., understand how the uncertainty propagates, and how the uncertainty affects the solution for large time. For Landau damping phenomena, uncertainty quantification (UQ) may help us make reliable predictions for the large time behavior of the solution.

To model the uncertainty, we introduce a random variable $z$ in a random space $I_z\subset \mathbb{R}^{d_z}$ with probability distribution $\pi(z)\rd{z}$, where $d_z$ is the dimension of the random space. Then the uncertainty from initial data is modeled by taking the initial data $f_{in}=f_{in}(x,v,z)$ to be $z$-dependent. Then we still consider the same equation (\ref{VP0}), with $f$ and $E$ depending on the extra random variable $z$. Now the problem becomes to investigate the $z$-dependence of $f(x,v,t,z)$ and $E(x,t,z)$ for large $t$, given a $z$-dependent initial data satisfying the conditions for Landau damping.

One of the important tasks in UQ is to study the sensitivity of the solution on the random inputs. With estimates for $z$-derivatives of $f$ and $E$, one can understand whether the solution is sensitive to the random perturbations, as well as help prove the spectral accuracy for generalized polynomial chaos (gPC) based numerical methods~\cite{XiuK,Xiu}, including stochastic collocation and stochastic Galerkin methods. Recently there has been a rapid progress in studying the random space regularity and spectral accuracy of numerical methods for kinetic equations, for both linear equations~\cite{JLM, LW,JL,Liu} and nonlinear equations~\cite{JZ,SJ,LJ} with random uncertainties. However, all of these works are based on energy estimates, taking advantage of hypocoercivity of the linearized kinetic operators (linearized collision operators, Fokker-Planck operators, etc.). Nonlinear terms are controlled by the hypocoercive terms together with the assumption that the initial data is near the global equilibrium. For the VP equation (\ref{VP0}), due to the time-reversibility, there cannot be a dissipative energy estimate. Therefore one has to go out of the framework of energy estimate in order to study the impact of random uncertainty.

This work is a first attempt towards the study of random uncertainty for the VP equation and Landau damping. Our analysis is based on the framework of~\cite{CM}. For simplicity, we will assume each of $x,v,z$ is one-dimensional. Denoting $f^*(x,v)$ as a prescribed time-asymptotic profile, \cite{CM} proved that there exists initial data such that the solution $f(t,x,v)$ satisfies $\lim_{t\rightarrow\infty} \|f(t,x,v)-f^*(x-vt,v)\|=0$ (in some sense), under smoothness and smallness assumptions on $f^*$. In other words, the solution $f(t,x,v)$ behaves like the free transport of $f^*(x,v)$ for large time. Now we assume that $f^*=f^*(x,v,z)$ also depends on $z$.  By generalizing the fixed-point argument in~\cite{CM} to estimate the derivatives of $E$ with respect to $x,v,z$, we are able to prove: if we denote $f(t,x,v,z)$ as the Landau damping solution with $\lim_{t\rightarrow\infty} \|f(t,x,v,z)-f^*(x-vt,v,z)\|=0$ obtained in~\cite{CM}, and assuming $f^*$ has smooth dependence on $z$, then $f(t_0,x,v,z)$ depends smoothly on $z$ if $t_0$ is large enough. Although not being able to prove the $z$-regularity of the solution for all small enough initial data, our result shows that there exists a class of initial data depending smoothly on $z$, such that this regularity is maintained for all time.

It is desirable to have a class of Landau damping solutions with uncertainty, such that the $z$-derivative of the solution of any order is controlled. Therefore one would desire the smallness condition on the initial data independent of $K$, the order of $z$-derivative. However, a direct extension of the contraction argument in~\cite{CM} will require the smallness condition on $f^*$ depending on $K$. To overcome this difficulty, we let $t_0$, the time when the estimate starts to work, increase with $K$. With $t_0=t_0(K)$, all other parameters appeared in the smallness condition can be made independent of $K$. This means that for a class of $f^*$, we can estimate the $z$-derivative of the solution of any order, but the estimates start to work at later time for higher order derivatives. This is less restrictive than the requirement of $f^*$ depending on $K$, since it is natural to expect an extension of a local-in-time estimate~\cite{Ior} to handle the time period $[0,t_0]$ (but this is out of the scope of this paper).

This paper is organized as follows: in Section 2 we introduce notations and our main result. In Section 3 we introduce some necessary lemmas, including some from \cite{CM} and some new ones. In Section 4 we estimate the $x,v$-derivatives of the particle trajectory $X(x,v,t),V(x,v,t)$, and $x$-derivatives of the electric field $E$. In Section 5 we conduct estimates for the $z$-derivatives of $X$ and $E$. In Section 6 we prove the $z$-regularity of the particle distribution $f$. The paper is concluded in Section 7.

\section{Notations and the main result}

From now on we will consider the VP equation (\ref{VP0}) with one-dimensional $x,v,z$, which can be written as:
\begin{equation}\label{VP}
\partial_t f + v\partial_x f + E\partial_v f = 0,\quad  \partial_x E = \rho-\rho_0,\quad \rho = \int f \rd{v}.
\end{equation}

\subsection{Notations}

Fix $a\ge 0$, $t_0>0$. For a function $F$ with variables $t$, $(x,t)$ or $(x,v,t)$, denote
\begin{equation}
\|F\|_{a,t_0} = \sup_{t\ge t_0}e^{at}\|F(\cdot,t)\|_{L^\infty},
\end{equation}
where the $L^\infty$ norm is taken over all variables except $t$, and the corresponding space of functions $L^\infty_{a,t_0}$. The continuous functions in $L^\infty_{a,t_0}$ form a closed subset, denoted by $\opC_{a,t_0}$. 

We also define $\opC_{a,t_0,k}$ by the norm 
\begin{equation}\label{at0k}
\|F\|_{a,t_0,k}= \sup_{t\ge t_0}t^{-k}e^{at}\|F(\cdot,t)\|_{L^\infty},
\end{equation}
for positive integers $k$, and
\begin{equation}
\|F\|_{L^\infty_k} = \|F\|_{0,t_0,k},
\end{equation}
for positive integers $k$.

\subsection{Summary of the main result in~\cite{CM}}
\cite{CM} considers the deterministic 1d VP equation (\ref{VP}) (without $z$-dependence). Given a time-asymptotic profile $f^*(x,v)$, their goal is to find a solution $f(t,x,v)$ to (\ref{VP}) such that $\lim_{t\rightarrow\infty} \|f(x,v,t)-f^*(x-vt,v)\|=0$ (in some sense). Since $g(x,v,t) = f^*(x-vt,v)$ satisfies the free transport equation $\partial_t g + v\partial_x g = 0$ with initial data $g(\cdot,\cdot,0) = f^*$, this goal is to say that for large time, the solution $f(x,v,t)$ behaves like the solution to the free transport equation with initial data $f^*$. In other words, this specific solution exhibits the behavior of Landau damping.

They assume $f^*$ satisfying
\begin{itemize}
\item ({\bf Smoothness})($a_1$): $\hat{f}^*(k_x,k_v) \le \frac{a_1}{1+k_x^2}e^{-a|k_v|}$,
\item ({\bf Decay})($a_2$): $|f^*(x,v)| \le \frac{a_2}{1+v^4}$,
\end{itemize}
for some positive constants $a,a_1,a_2$, with $\hat{f}^*$ being the Fourier transform in both $x$ and $v$:
\begin{equation}
\hat{f}^*(k_x,k_v) = \frac{1}{2\pi} \int \int f^*(x,v) e^{i(k_x x + k_v v)}\rd{v}\rd{x}.
\end{equation}
The first assumption basically says that $f^*$ is analytic in $v$, and second-order differentiable in $x$. For example, it is straightforward to see that $f^*(x,v) = f_1(x)e^{-v^2}$ satisfies this assumption if $f_1$ is small enough in $H^2_x$. Note that this $f^*$ also satisfies the second assumption.

\cite{CM} assumes the following conditions on the constants:
\begin{itemize}
\item $a \ge 15\sqrt{a_2}$,
\item $t_0 \ge \max(0,\frac{\log(8a_1)}{a})$,
\end{itemize}
and showed that for $f^*(x,v)$ satisfying the above conditions, there exists initial data at $t=t_0$ such that the solution to (\ref{VP}) satisfies 
\begin{equation}\label{fst}
f(x,v,t)\approx f^*(x-vt,v),
\end{equation}
for large $t$. The key idea is the following map $\opF$ (Lemma 3.1 in \cite{CM}), which maps a field $F(x,t)$ with $\|F\|_{a,t_0}e^{-at_0}\le a$ and satisfying the Lipschitz condition into:
\begin{itemize}
\item Define the particle trajectory $X=X(x,v,t),\,V=V(x,v,t)$ by
\begin{equation}\label{traj}
\begin{split}
\dot{X} & = V, \\
\dot{V} & = F(X,t), \\
\lim_{t\rightarrow\infty} X-Vt & = x, \\
\lim_{t\rightarrow\infty} V &  = v.
\end{split}
\end{equation}
The Hamiltonian map from $(x,v)$ to $(X,V)$ is denoted as $\Phi_t$.
\item Define $f$ by 
\begin{equation}\label{deff}
f(X,V,t) = f^*(x,v),
\end{equation}
 i.e., $f$ solves the Liouville equation
\begin{equation}
\partial_t f + v\partial_x f + F\partial_v f = 0.
\end{equation}
\item Define $\opF(F)$ by $\partial_x \opF(F) = \rho-\rho_0$, with $\rho = \int f \rd{v}$. In other words,
\begin{equation}
\opF(F)(x,t) = \int B(x-y)\rho(y,t) \rd{y},
\end{equation}
with the convolution kernel $B$ given by
\begin{equation}
B(x)=\frac{1}{2}-\frac{x}{2\pi}\quad \text{for}\quad x\in [0,2\pi),\quad B(x+2\pi)=B(x).
\end{equation}
\end{itemize}
Then to solve (\ref{VP}) with the asymptotic limit (\ref{fst}) is equivalent to finding a fixed point of $\opF$.

\cite{CM} proved that $\opF$ maps the set with small $\opC_{a,t_0}$ norm and the Lipschitz condition (which is a closed subset of $\opC_{a,t_0}$) into itself, and is contractive. This provides a fixed point of $\opF$. From now on, we will denote $E$ as the unique fixed point of $\opF$. When we consider (\ref{VP}) with $z$-dependence, the result of \cite{CM} can be applied for each fixed $z$, i.e., if $f^*$ depends on $z$, then $E$ also depends on $z$.

\subsection{Our main result}

Fix a positive integer $K$, and we will estimate $\partial_z^K E$. First we assume that $f^*$ satisfies ({\bf Smoothness})($a_1$) and ({\bf Decay})($a_2$) for each fixed $z$, with $\nabla_{x,v} f^*$ satisfying ({\bf Decay})($a_2$). We further assume that all the $x,v,z$-derivatives of $f^*$ up to total order $K$ satisfy ({\bf Smoothness})($C$) and ({\bf Decay})($C$) for some constant $C$. For example, it is straightforward to see that $f^*(x,v,z) = f_1(x,z)e^{-v^2}$ satisfies these assumptions if $f_1$ is in $H^{K+2}_{x,z}$, and is small enough in $H^2_x$ for each fixed $z$.

Next we assume the constants $a,a_1,a_2,t_0$ satisfying the following conditions:
\begin{itemize}
\item ({\bf A1}) $a \ge \max\{1,15\sqrt{a_2}\}$,
\item ({\bf A2}) $t_0 \ge \max\{2,4K,\frac{1}{a}\log(8a_1)\}$,
\item ({\bf A3}) $\frac{50C_E}{a} (3/a)^3e^{-3} \le 1$, which implies $\frac{50C_E}{a} t_0^3e^{-at_0} \le 1$. Here $C_E=\frac{240a_1a_2}{a} + 4a_1$,
\item ({\bf A4}) $8e \le \frac{1}{20a_2}$,
\item ({\bf A5}) $8C_E \le a^2$.
\end{itemize}
These conditions are clearly satisfied if one first chooses $a_1$ and $a_2$ small enough, then $a$ and $t_0$ large enough.

Then we have
\begin{theorem}\label{thm}
Under the aforementioned assumptions ({\bf A1})-({\bf A5}) and the ({\bf Smoothness}) and ({\bf Decay}) of $f^*$ and its derivatives, $E(z)$, the fixed point of $\opF$ given by \cite{CM}, satisfies the estimate
\begin{equation}
\|\partial_z^k E\|_{a,t_0} \le C,\quad 0\le k \le K.
\end{equation}
\end{theorem}

\begin{corollary}\label{cor}
Under the same assumptions, $f(x,v,t,z)$, the Landau damping solution given by \cite{CM}, satisfies
\begin{equation}
\|\partial_z^k[f(x,v,t,z)-f^*(x-vt,v,z)]\|_{a,t_0,1}\le C,\quad 0\le k \le K.
\end{equation}
\end{corollary}

Theorem \ref{thm} means that when time is large, the electric field of the Landau damping solution given by \cite{CM} is insensitive to the random perturbation on the initial data. Corollary \ref{cor} means that when time is large, the $z$-dependence of the particle distribution $f$ is dominated by the $z$-dependence of the time-asymptotic profile $f^*$ and insensitive to the uncertainty propagated from the electric field.

To prove Theorem \ref{thm}, we will use induction on $k$. The case $k=1$ can be proved by simply adopting estimates for the first order $x,v$-derivatives of $X,V$ and $x$-derivative of $E$. For larger $k$, we have to involve higher order $x,v$-derivatives to close the estimate. Therefore we start by estimating the higher order $x,v$-derivatives of $X,V$ and $x$-derivatives of $E$. It is important to adopt the correct norms in this part ((\ref{at0k}) with proper index $k$), since taking $x,v$-derivatives will deteriorate the time decay by polynomial orders. The simplest case is already noted in \cite{CM}, which proved that $\|\partial_x E\|_{b,t_0} \le C$ for all $b<a$. 

%
Our method is based on the fixed-point argument from \cite{CM}. However, one has to pay attention to the following facts:
\begin{enumerate}
\item When taking $x,v$-derivatives, the 'self-interacting' term contains more and more terms as the order of derivative increase, which makes it harder to have norm less than 1/2. In order to make $a,a_1,a_2$ independent of $K$ (which means that the initial data does not shrink to zero as $K\rightarrow\infty$), we let $t_0$ depends on $K$, see ({\bf A2}), which means that for higher order $z$-derivatives, our estimate starts to work at later time. This is less restrictive than the requirement of $a,a_1,a_2$ depending on $K$, since one may extend the estimate to earlier time based on local-in-time estimates. For example, one can estimate $\partial_z^k E$ based on the existence theorem for the deterministic VP equation~\cite{Ior}, but this is out of the scope of this paper.
\item The self-interacting term contains $x,v$-derivatives of order no more than ONE. As a result, all the higher order $x,v$-derivatives and $z$-derivatives of $f^*$ are only required to satisfy ({\bf Smoothness}) and ({\bf Decay}), with arbitrarily large constants.
\end{enumerate}

\section{Preliminaries}

\subsection{Summary of intermediate results from \cite{CM}}
Apart from the contraction property of $\opF$, we will use some intermediate results from \cite{CM}. We first list a few estimates in \cite{CM}:
\begin{itemize}
\item $\|E\|_{a,t_0} \le 8a_1$,
\item $\|\partial_x E\|_{L^\infty} \le 20a_2$ (a consequence of (\ref{rhoest1}) below),
\item $\|\partial_x E\|_{a,t_0,1} \le C_E$ (a consequence of (\ref{rhoest2}) below, with $C_E=\frac{240a_1a_2}{a} + 4a_1$ given explicitly).
\end{itemize}
We remark that \cite{CM} claims $\|E\|_{a,t_0} \le 8a_1a_2$, but there is a calculation error on page 319 of \cite{CM}. The correct estimate for $\opF(0)(x,t)$ goes as
\begin{equation}
\begin{split}
|\opF(0)(x,t) | = & \left| \sum_{k\ne 0} \frac{\hat{f}^*(k,kt)e^{ikx}}{ik} \right| \le \sum_{k\ne 0} \frac{|\hat{f}^*(k,kt)|}{|k|} \le \sum_{k\ne 0} \frac{a_1}{(1+k^2)|k|}e^{-a|k|t} \\ \le &  e^{-at}\sum_{k\ne 0} \frac{a_1}{(1+k^2)|k|} \le 4a_1e^{-at},
\end{split}
\end{equation}
and then the contraction property of $\opF$ implies $\|E\|_{a,t_0} \le 8a_1$.

\begin{lemma}
\begin{equation}\label{varchange}
\int \phi(y,u)f(y,u,t) \rd{y}\rd{u} = \int \phi(X(x,v,t),V(x,v,t))f^*(x,v)\rd{x}\rd{v}.
\end{equation}
In particular,
\begin{equation}\label{F}
E(y,t) = \int B(y-X(x,v,t))f^*(x,v)\rd{x}\rd{v}.
\end{equation}
\end{lemma}
The above lemma is Lemma 3.1 Step 3 in  \cite{CM}.
\begin{lemma}\label{lem_rhoest}
If $g^*$ satisfies ({\bf Decay})($c_2$), and define ('the density given by $g^*$')
\begin{equation}
\rho_g(y,t) = \int \delta(y-X)g^*(x,v)\rd{x}\rd{v},
\end{equation}
then one has
\begin{equation}\label{rhoest1}
\|\rho_g\|_{L^\infty} \le 10c_2.
\end{equation}
If $g^*$ satisfies ({\bf Smoothness})($c_1$) and ({\bf Decay})($c_2$), with $\nabla_{x,v} g$ satisfying ({\bf Decay})($c_2$), then
\begin{equation}\label{rhoest2}
\|\rho_g-\rho_{g,0}\|_{a,t_0,1} \le C = \frac{240a_1c_2}{a} + 4c_1,
\end{equation}
where $\rho_{g,0} = \int g^*(x,v)\rd{x}\rd{v}$.
\end{lemma}
\begin{proof}
This is Lemma 3.1 Step 4 and Theorem 3.4 Step 2 in  \cite{CM}, with slight improvement. We include the proof below.

First, one has the fact that $|v-V(x,v,t)| \le (\|E\|_{a,t_0}/a)e^{-at}$ (equation (A.2) in \cite{CM}). By assumption one has $|g^*(x,v)| \le c_2/(1+v^4)$. Thus
\begin{equation}
\begin{split}
|g^*(x,v)| \le c_2,\quad & \text{for}\quad  |v| < (\|E\|_{a,t_0}/a)e^{-at}, \\
|g^*(x,v)| \le \frac{c_2}{1+(|V|-(\|E\|_{a,t_0}/a)e^{-at})^4},\quad & \text{for} \quad |v| \ge (\|E\|_{a,t_0}/a)e^{-at}, \\
\end{split}
\end{equation}
where the second inequality is because of the condition $|v|\ge |V|-(\|E\|_{a,t_0}/a)e^{-at}$. Then, writing $g(X,V,t) = g^*(x,v)$,
\begin{equation}
\begin{split}
|\rho_g(x,t)| = &  \int \int \delta(y-X) g(X,V,t)\rd{X}\rd{V} \\
\le & \int \int_{|V|<2(\|E\|_{a,t_0}/a)e^{-at}} \delta(y-X)c_2\rd{V}\rd{X} \\ & + \int \int_{|V|\ge 2(\|E\|_{a,t_0}/a)e^{-at}} \delta(y-X)\frac{c_2}{1+(|V|-(\|E\|_{a,t_0}/a)e^{-at})^4}\rd{V}\rd{X} \\
\le & 4(\|E\|_{a,t_0}/a)e^{-at}c_2 + \int \frac{c_2}{1+(V/2)^4}\rd{V}\\ 
\le & 4c_2 + \sqrt{2}\pi c_2 \le 10 c_2,\\ 
\end{split}
\end{equation}
where we used the fact that $|V|\ge 2(\|E\|_{a,t_0}/a)e^{-at}$ implies $|v| \ge (\|E\|_{a,t_0}/a)e^{-at} $ in the first inequality, and the fact that $(\|E\|_{a,t_0}/a)e^{-at} \le 1$ (due to ({\bf A2})) in the last step. This proves (\ref{rhoest1}).

Next, fix $t$,
\begin{equation}
\begin{split}
|\rho_g(X,t)-\rho_{g,0}| = &  \left|\int g^*(x(X,V),v(X,V))\rd{V} - \rho_{g,0}\right| \\
\le & \int |g^*(x(X,V),v(X,V)) - g^*(X-Vt,V)|\rd{V} + \left|\int g^*(X-Vt,V)\rd{V} -\rho_{g,0}\right| \\
= & I_1+I_2.
\end{split}
\end{equation}
$I_1$ is estimated by
\begin{equation}
\begin{split}
I_1 = & \int |\nabla g^*(x_1,v_1)| \cdot(|x-(X-Vt)| + |v-V|) \rd{V} \\ &  \text{where} \quad (x_1,v_1) \text{ are some values with } |v_1|\ge \min\{|v|,|V|\}\\
\le & \int \frac{c_2}{1+\max\{0,|V|-1\}^4} \cdot(  \|E\|_{a,t_0}\frac{2}{a}te^{-at}+ (\|E\|_{a,t_0}/a)e^{-at}) \rd{V} \\
\le & c_2 (2+\sqrt{2}\pi) \cdot 8a_1 \frac{3}{a} te^{-at} \le \frac{240a_1c_2}{a}te^{-at},
\end{split}
\end{equation}
where in the first inequality we used $|v-V(x,v,t)| \le (\|E\|_{a,t_0}/a)e^{-at}\le 1$ and
\begin{equation}
\begin{split}
|x-(X-Vt)| = &\left|\int_t^\infty s E(X(s),s) \rd{s}\right| \le \|E\|_{a,t_0} \int_t^\infty se^{-as} \rd{s} = \|E\|_{a,t_0}(\frac{1}{a}te^{-at} + \frac{1}{a^2} e^{-at}) \\ \le & \|E\|_{a,t_0}\frac{2}{a}te^{-at}. \\
\end{split}
\end{equation}
To estimate $I_2$, first notice that
\begin{equation}
\int g^*(X-Vt,V) \rd{V} = \sum_k \hat{g}^*(k,kt)e^{ikX},\quad \rho_{g,0} = \hat{g}^*(0,0).
\end{equation}
Then $I_2$ is estimated by 
\begin{equation}
\begin{split}
I_2 \le \sum_{k\ne 0} |\hat{g}^*(k,kt)| \le \sum_{k\ne 0} \frac{c_1}{1+k^2}e^{-a|k|t} \le e^{-at}c_1\sum_{k\ne 0} \frac{1}{1+k^2} \le 4c_1e^{-at}.
\end{split}
\end{equation}
This finishes the proof of (\ref{rhoest2}) with $C = \frac{240a_1c_2}{a} + 4c_1$.
\end{proof}

The following lemma is a modification of Lemma 3.1 Step 7, 8 in \cite{CM}.
\begin{lemma}\label{lem_F}
Define $\opF(E;g^*)$ by
\begin{equation}
\opF(E;g^*) = \int B(y-X)g^*(x,v)\rd{x}\rd{v}.
\end{equation}
If $g^*$ satisfies ({\bf Smoothness})($c_1$) and ({\bf Decay})($c_2$), then
\begin{equation}
\|\opF(E;g^*)\|_{a,t_0} \le C = 8a_1\frac{88c_2}{a^2-80a_2} + 4c_1.
\end{equation}
\end{lemma}
\begin{proof}
Lemma 3.1 Step 7 in \cite{CM}  with (\ref{rhoest1}) implies that
\begin{equation}
\|\opF(E_1;g^*)-\opF(E_2;g^*)\|_{a,t_0} \le \frac{88c_2}{a^2-80a_2}\|E_1-E_2\|_{a,t_0}.
\end{equation}
Step 8 (corrected) implies that
\begin{equation}
\|\opF(0;g^*)\|_{a,t_0} \le 4c_1.
\end{equation}
Then the conclusion follows from the estimate $\|E\|_{a,t_0} \le 8a_1$.
\end{proof}

\subsection{Contraction property in $\opC_{a,t_0,k}$}
In order to estimate the derivatives of $X,V,E$ in the spaces $\opC_{a,t_0,k}$, we will need the contraction properties in these spaces. The assumption $t_0\ge 4K$ in ({\bf A2}) will play a crucial role in making the resulting smallness conditions independent of $K$.

We first compute the integrals
\begin{equation}\label{inte}
\begin{split}
\int_t^\infty s^ke^{-as} \rd{s} = & \frac{1}{a}t^ke^{-at} + \frac{k}{a^2}t^{k-1}e^{-at} + \cdots + \frac{k!}{a^{k+1}}e^{-at} = e^{-at}\sum_{j=0}^k \frac{k!}{j!a^{k-j+1}}t^j, \\
\int_t^\infty (s-t)s^ke^{-as} \rd{s} = & e^{-at}\left[\sum_{j=0}^{k+1} \frac{(k+1)!}{j!a^{k+1-j+1}}t^j  - \sum_{j=0}^k \frac{k!}{j!a^{k-j+1}}t^{j+1} \right]\\
= & e^{-at}\left[ \frac{(k+1)!}{a^{k+2}} + t\sum_{j=0}^{k} \left(\frac{(k+1)!}{(j+1)!}-\frac{k!}{j!}\right)\frac{1}{a^{k-j+1}}t^j \right] \\
= & e^{-at}\left[ \frac{(k+1)!}{a^{k+2}} + t\sum_{j=0}^{k-1} \left(\frac{(k+1)!}{(j+1)!}-\frac{k!}{j!}\right)\frac{1}{a^{k-j+1}}t^j \right]. \\
\end{split}
\end{equation}
Then one has
\begin{equation}\label{int1}
\int_t^\infty (s-t)s^ke^{-as} \rd{s} \le t^ke^{-at}\frac{1}{a^2}\left[ \frac{(k+1)!}{t_0^k} +\sum_{j=0}^{k-1}\left(\frac{(k+1)!}{(j+1)!}-\frac{k!}{j!}\right)\frac{1}{t_0^{k-1-j}}\right],
\end{equation}
if $t \ge t_0$ and assumption ({\bf A1}) providing $a\ge 1$. Thus
\begin{equation}
\left(\frac{(k+1)!}{(j+1)!}-\frac{k!}{j!}\right)\frac{1}{t_0^{k-1-j}} = \frac{k!(k-j)}{(j+1)!}\frac{1}{t_0^{k-1-j}} \le (k-j)\frac{k^{k-1-j}}{t_0^{k-1-j}} \le 2(k-j)2^{-(k-j)},
\end{equation}
for integer $0\le j \le k-1$, since ({\bf A2}) gives $t\ge t_0 \ge 4K \ge 2k$. Similarly $\frac{(k+1)!}{t_0^k} \le 2^{-k+1} $. Thus the quantity in the bracket in (\ref{int1}) is controlled by
\begin{equation}
2^{-k+1} + \sum_{j=0}^{k-1} 2(k-j)2^{-(k-j)} = 2^{-k+1} + 4-2(k+2)2^{-k} < 4.
\end{equation}
Therefore we obtain
\begin{lemma}\label{lem_op}
Under the assumption ({\bf A1}), ({\bf A2}) and $t\ge t_0,\,k\le 2K$,
\begin{equation}
\int_t^\infty (s-t)s^ke^{-as} \rd{s} \le t^ke^{-at}\frac{4}{a^2}.
\end{equation}
\end{lemma}
By estimating the first integral in (\ref{inte}) a similar way, one obtains
\begin{lemma}\label{lem_op1}
Under the assumption ({\bf A1}), ({\bf A2}) and $t\ge t_0,\,k\le 2K$,
\begin{equation}
\int_t^\infty s^ke^{-as} \rd{s} \le t^ke^{-at}\frac{2}{a}.
\end{equation}
\end{lemma}

Now we give the contraction argument in $\opC_{a,t_0,k}$, which can be viewed as a generalization of the contraction argument in \cite{CM}:
\begin{lemma}\label{lem_contra}
Consider the operator
\begin{equation}
P[Y](t) = \int_t^\infty (s-t) (\partial_x E(X(x,v,s),s)Y(s)+\Phi(s))  \rd{s},
\end{equation}
where $\Phi(s)$ is some given source term. Then under assumptions ({\bf A1}), ({\bf A2}),
\begin{equation}\label{contra1}
\|P[Y_1](t) - P[Y_2](t)\|_{a,t_0,k} \le \frac{80a_2}{a^2}\|Y_1-Y_2\|_{a,t_0,k},
\end{equation}
for $k\le 2K$ and any $Y_1,Y_2\in \opC_{a,t_0,k}$. In particular, if $\Phi(s) \in \opC_{a,t_0,k}$, then $P$ is a contraction map on $\opC_{a,t_0,k}$, and one has the estimate
\begin{equation}\label{contra2}
\|Y_0\|_{a,t_0,k} \le \frac{1}{1-\frac{80a_2}{a^2}}\frac{4}{a^2}\|\Phi\|_{a,t_0,k} \le \frac{8}{a^2}\|\Phi\|_{a,t_0,k},
\end{equation} 
for the unique fixed point $Y_0$ of $P$.
\end{lemma}
\begin{proof}
\begin{equation}
\begin{split}
|P[Y_1](t) - P[Y_2](t)| & \le \int_t^\infty (s-t) |\partial_x E(X(x,v,s),s)|\cdot|Y_1(s)-Y_2(s)|  \rd{s} \\
& \le \|\partial_x E\|_{L^\infty}\|Y_1-Y_2\|_{a,t_0,k}\int_t^\infty (s-t)s^k e^{-as}  \rd{s} \\
& \le 80a_2\|Y_1-Y_2\|_{a,t_0,k}\frac{t^ke^{-at}}{a^2},  \\
\end{split}
\end{equation}
thus one gets (\ref{contra1}) (where the last step is due to Lemma \ref{lem_op}). If $k\le 2K$, then the constant in (\ref{contra1}) is at most $1/2$, by ({\bf A1}). In this case, since
\begin{equation}
P[0] \le  \|\Phi\|_{a,t_0,k}\int_t^\infty (s-t)s^k e^{-as}\rd{s} \le \|\Phi\|_{a,t_0,k}\frac{4t^ke^{-at}}{a^2},
\end{equation}
i.e.,
\begin{equation}
\|P[0]\|_{a,t_0,k} \le \|\Phi\|_{a,t_0,k}\frac{4}{a^2},
\end{equation}
then
\begin{equation}
\begin{split}
\|Y_0\|_{a,t_0,k} & = \|P[Y_0]\|_{a,t_0,k} \le  \|P[Y_0]-P[0]\|_{a,t_0,k} + \|P[0]\|_{a,t_0,k} \\ 
& \le \frac{80a_2}{a^2}\|Y_0\|_{a,t_0,k}+ \|\Phi\|_{a,t_0,k}\frac{4}{a^2},
\end{split}
\end{equation}
which implies (\ref{contra2}).
\end{proof}
\begin{remark}
Here we explain the importance of the assumption $t_0\ge 4K$. If on the contrary, we take $t_0$ to be a fixed constant. Then the bracket in (\ref{int1}) will be at least $O((k+1)!)$, and as a result, the constant in (\ref{contra1}) will be at least $O((k+1)!)$. Since one needs this constant to be at most $1/2$ to obtain a contraction map, one will need $a_2\le \frac{C}{(k+1)!}a^2$. This will prevent the assumptions on $a,a_1,a_2$ being independent of $K$.
\end{remark}

\subsection{Formulas for higher order derivatives}
We need a few lemmas regarding the higher derivatives of composite functions, which are variants of the Fa\`a di Bruno formula. It is easy to prove them by induction, and we omit the proof. 
\begin{lemma}\label{lem_dzE}
Let $E = E(X,z), X=X(z)$. Then
\begin{equation}
\frac{\partial^k}{\partial z^k}[E(X(z),z)] = \sum_{\alpha+\beta+\gamma= k} \left[ \partial_x^\alpha\partial_z^\beta E \sum_{\gamma_1+\dots+\gamma_\alpha=\gamma}c\prod_{j=1}^\alpha \partial_z^{\gamma_j+1}X \right],
\end{equation}
where all the indices are non-negative, and the constants $c$ are non-negative integers depending on the summation indices (we suppress this dependence). It can be written as a polynomial
\begin{equation}
\frac{\partial^k}{\partial z^k}[E(X(z),z)] = P^E_k\left(\{\partial_x^\alpha\partial_z^\beta E\}_{\alpha+\beta\le k,\beta\le k-1},\{\partial_z^\delta X\}_{1\le\delta\le k-1}\right) + \partial_x E \partial_z^k X + \partial_z^k E.
\end{equation} 
\end{lemma}
\begin{lemma}\label{lem_dzB}
Let $B = B(X), X=X(z), g=g(z)$. Then
\begin{equation}
\frac{\partial^k}{\partial z^k}[g(z)B(X(z))] = \sum_{\alpha+\beta+\gamma= k} \left[ \partial_z^\beta g\partial_x^\alpha B \sum_{\gamma_1+\dots+\gamma_\alpha=\gamma}c\prod_{j=1}^\alpha \partial_z^{\gamma_j+1}X \right].
\end{equation}
It can be written as a polynomial
\begin{equation}
\frac{\partial^k}{\partial z^k}[g(z)B(X(z))] = P^B_k\left(\{\partial_x^\alpha B\}_{\alpha\le k},\{\partial_z^\beta g\}_{\beta\le k},\{\partial_z^\delta X\}_{1\le\delta\le k-1}\right) + g\partial_x B \partial_z^k X.
\end{equation} 
\end{lemma}
\begin{remark}
In Lemma \ref{lem_dzE} and Lemma \ref{lem_dzB}, it is important to notice that $P^E_k$ and $P^B_k$ do NOT involve $\partial_z^k X$, and for those monomials with $\alpha>0$, there is at least one factor like $\partial_z^\delta X$ in it.
\end{remark}
\begin{lemma}\label{lem_dXx}
Fix $t$, and consider the map $\Phi_t: (x,v)\mapsto(X(x,v,t),V(x,v,t))$. Let $g=g(x,v)$. Then
\begin{equation}\label{dXx}
\frac{\partial^k}{\partial X^k}[g(x(X,V),v(X,V))] = \sum_{\alpha+\beta\ge 1,\,\alpha+\beta+ \sum (\gamma_j + \delta_j-1) = k} c \partial_x^\alpha \partial_v^\beta g \prod_{j=1}^k \partial_x^{\gamma_j}\partial_v^{\delta_j} V,
\end{equation}
where the RHS is evaluated at $(x(X,V),v(X,V))$. It can be written as a polynomial
\begin{equation}
\frac{\partial^k}{\partial X^k}[g(x(X,V),v(X,V))] = P^g_k + P^{g,1}_k + P^{g,0}_k,
\end{equation} 
where 
\begin{equation}
P^{g,0}_k = \left(\frac{\partial V}{\partial v}\right)^k\partial_x^k g,
\end{equation}
and
\begin{equation}\label{pg1}
P^{g,1}_k = \left(\frac{\partial V}{\partial v}\right)^{k-1}\sum_{j=1}^k\partial_x^{k-j}\left[(j-1)\left(\frac{\partial^2 V}{\partial v\partial x}\partial_x^{j-1} g\right) - \left(\frac{\partial V}{\partial x}\partial_x^{j-1} \partial_v g\right)\right],
\end{equation}
includes all the terms with all but one factor in the product $\prod_{j=1}^k \partial_x^{\gamma_j}\partial_v^{\delta_j} V$ not being $\frac{\partial V}{\partial v}$, and $P^g_k$  includes all the terms with at least two factors in the product not being $\frac{\partial V}{\partial v}$. 
\end{lemma}
\begin{proof}
First notice that when considering the LHS of (\ref{dXx}),
\begin{equation}
\frac{\partial}{\partial X} = \frac{\partial x}{\partial X}\frac{\partial}{\partial x}+\frac{\partial v}{\partial X}\frac{\partial}{\partial v}= \frac{\partial V}{\partial v}\frac{\partial}{\partial x}-\frac{\partial V}{\partial x}\frac{\partial}{\partial v},
\end{equation}
where the last equality uses the Hamiltonian structure of $\Phi_t$ defined in (\ref{traj}). Thus
\begin{equation}\label{dkg}
\frac{\partial^k g}{\partial X^k} =\left( \frac{\partial V}{\partial v}\frac{\partial}{\partial x}-\frac{\partial V}{\partial x}\frac{\partial}{\partial v}\right)^k g.
\end{equation}
This gives the structure on the RHS of (\ref{dXx}), since each time in a monomial, one factor gets a $x$ or $v$ derivative, and the whole monomial is multiplied by $\frac{\partial V}{\partial v}$ or $\frac{\partial V}{\partial x}$. The relation $\alpha+\beta \ge 1$ is because at the beginning the derivative has to hit $g$.

It is clear that if all the factors in the product $\prod_{j=1}^k \partial_x^{\gamma_j}\partial_v^{\delta_j} V$ are $\frac{\partial V}{\partial v}$, then the only possibility is to choose $ \frac{\partial V}{\partial v}\frac{\partial}{\partial x}$ in all of the $k$ operators in (\ref{dkg}) and the $x$-derivatives always hit $g$. This gives the term $P^{g,0}_k = (\frac{\partial V}{\partial v})^k\partial_x^k g$. 

For the terms with all but one factor in the product not being $\frac{\partial V}{\partial v}$, there are two possibilities: to choose $ \frac{\partial V}{\partial v}\frac{\partial}{\partial x}$ in all the $k$ operators in (\ref{dkg}) with one derivative not hitting $g$ or the factor not being $\frac{\partial V}{\partial v}$; or to choose $ \frac{\partial V}{\partial v}\frac{\partial}{\partial x}$ for $k-1$ operators and $\frac{\partial V}{\partial x}\frac{\partial}{\partial v}$ for one operator with all derivatives hitting $g$ or the factor not being $\frac{\partial V}{\partial v}$. This gives
\begin{equation}
P^{g,1}_k = \sum_{j=1}^k(j-1)\left(\frac{\partial V}{\partial v}\right)^{k-1}\partial_x^{k-j}\left(\frac{\partial^2 V}{\partial v\partial x}\partial_x^{j-1} g\right) - \sum_{j=1}^{k}\left(\frac{\partial V}{\partial v}\right)^{k-1}\partial_x^{k-j}\left(\frac{\partial V}{\partial x}\partial_x^{j-1} \partial_v g\right),
\end{equation}
where the first term is the first possibility with the $j$-th derivative not hitting $g$, and the second term is the second possibility with the $j$-th operator chosen as $\frac{\partial V}{\partial x}\frac{\partial}{\partial v}$. This gives (\ref{pg1}).
\end{proof}

\subsection{Nonlinear estimate in $\opC_{a,t_0,k}$}
We prove a nonlinear estimate in the spaces $\opC_{a,t_0,k}$ with various $k$ values:
\begin{lemma}\label{lem_nl}
Let $k,k_1,\dots,k_n,\,n\ge 2$ be nonnegative integers, and $f_i \in \opC_{a,t_0,k_i},\,i=1,\dots,n$. Then
\begin{equation}
\left\|\prod_{i=1}^n f_i\right\|_{a,t_0,k} \le C\prod_{i=1}^n \|f_i\|_{a,t_0,k_i},
\end{equation}
with $C=t_0^{\sum k_i - k}e^{-(n-1)at_0}$ if $t_0\ge t_1:=\frac{\sum k_i - k}{(n-1)a}$, and $C=t_1^{\sum k_i - k}e^{-(n-1)at_1}$ otherwise.
\end{lemma}
\begin{proof}
\begin{equation}
\begin{split}
\left\|\prod_{i=1}^n f_i\right\|_{a,t_0,k} = & \sup_{t\ge t_0} t^{-k}e^{-at}\prod_{i=1}^n f_i =  \sup_{t\ge t_0}t^{\sum k_i - k}e^{-(n-1)at}\prod_{i=1}^n ( t^{-k_i}e^{-at}f_i) \\ \le & \prod_{i=1}^n \|f_i\|_{a,t_0,k_i} \sup_{t\ge t_0}t^{\sum k_i - k}e^{-(n-1)at}.
\end{split}
\end{equation}
The function $\phi(t)=\sup_{t\ge t_0}t^{\sum k_i - k}e^{-(n-1)at}$ attains its maximum in $[0,\infty)$ at $t_1=\frac{\sum k_i - k}{(n-1)a}$, and is monotone in $[0,t_1]$ and $[t_1,\infty)$ respectively. Thus in the case $t_1\le t_0$, the maximum in $[t_0,\infty)$ is attained at $t=t_0$, otherwise at $t=t_1$.
\end{proof}

\section{$x,v$-derivatives of $X,V$ and $x$-derivatives of $E$}
Equation (A.1) in \cite{CM}  gives
\begin{equation}\label{XV}
\begin{split}
& X(x,v,t) = x + vt + \int_t^\infty (s-t) E(X(x,v,s),s)  \rd{s}, \\
& V(x,v,t) = v - \int_t^\infty  E(X(x,v,s),s)  \rd{s}.
\end{split}
\end{equation}
Taking $x,v$-derivatives,
\begin{equation}\label{dxv}
\begin{split}
& \frac{\partial X}{\partial x}(x,v,t) = 1 + \int_t^\infty (s-t) \partial_x E(X(x,v,s),s)\frac{\partial X}{\partial x}(x,v,s)  \rd{s}, \\
& \frac{\partial X}{\partial v}(x,v,t) =  t + \int_t^\infty (s-t) \partial_x E(X(x,v,s),s)\frac{\partial X}{\partial v}(x,v,s)  \rd{s}, \\
& \frac{\partial V}{\partial x}(x,v,t) = - \int_t^\infty \partial_x E(X(x,v,s),s)\frac{\partial X}{\partial x}(x,v,s)  \rd{s}, \\
& \frac{\partial V}{\partial v}(x,v,t) = 1 - \int_t^\infty \partial_x E(X(x,v,s),s)\frac{\partial X}{\partial v}(x,v,s)  \rd{s}. \\
\end{split}
\end{equation}
Notice that since (\ref{traj}) is a Hamiltonian system, the Jacobian $\det(\frac{\partial(X,V)}{\partial(x,v)})=1$.
\subsection{Estimate for first order $x,v$-derivatives of $X,V$}\label{sec_XV}

For fixed $x,v$, applying Lemma \ref{lem_contra} to (\ref{dxv}) with $\Phi(s)=\partial_x E(X(s),s),\partial_x E(X(s),s)s$ and $k=1,2$ respectively, yields the estimates
\begin{equation}
\begin{split}
& \left\| \frac{\partial X}{\partial x}(x,v,\cdot)-1\right\|_{a,t_0,1}  \le \frac{8}{a^2}C_E,\\
& \left\| \frac{\partial X}{\partial v}(x,v,\cdot)-t\right\|_{a,t_0,2}  \le \frac{8}{a^2}C_E.
\end{split}
\end{equation}
Notice that these estimates are independent of $x,v$. In particular, one gets
\begin{equation}
\left\|\frac{\partial X}{\partial x}\right\|_{L^\infty} \le 1 + \frac{8}{a^2}C_Et_0e^{-at_0} \le 2,\quad \left\|\frac{\partial X}{\partial v}\right\|_{L^\infty_1} \le 1 + \frac{8}{a^2}C_Et_0e^{-at_0} \le 2,
\end{equation}
 by ({\bf A5}), and the fact that $te^{-at}$ achieves its maximum in $[t_0,\infty)$ at $t_0$, and the maximum is no greater than 1 (a consequence of ({\bf A1}), ({\bf A2})). By the third and fourth equations of (\ref{dxv}), one easily deduces that 
\begin{equation}
\begin{split}
&\left|\frac{\partial V}{\partial x}\right| \le \left\|\frac{\partial X}{\partial x}\right\|_{L^\infty}\|\partial_x E\|_{a,t_0,1}\int_t^\infty se^{-as}\rd{s} \le 4\frac{C_E}{a}te^{-at},\\ 
&  \left|\frac{\partial V}{\partial v}-1\right| \le \left\|\frac{\partial X}{\partial v}\right\|_{L^\infty_1}\|\partial_x E\|_{a,t_0,1}\int_t^\infty s^2e^{-as}\rd{s} \le 10\frac{C_E}{a}t^2e^{-at},\\ 
\end{split}
\end{equation}
which implies the estimates
\begin{equation}
\begin{split}
&\left\|\frac{\partial V}{\partial x}\right\|_{a,t_0,1}  \le \frac{4C_E}{a},\\ 
&  \left\|\frac{\partial V}{\partial v}-1\right\|_{a,t_0,2}\le \frac{10C_E}{a}.\\ 
\end{split}
\end{equation}

\subsection{Higher $x,v$-derivatives of $X,V$ and $x$-derivatives of $E$}
Since we already have the estimate
$
\|\partial_x E\|_{a,t_0,1} \le C_E
$,
 we proceed to derivatives of order at least two. We will use induction on $k\ge 2$ to prove:
\begin{align}
& \label{induc1_1} \|\partial_x^k E \|_{a,t_0,k} \le C,\\
& \label{induc1_2} \|\partial_x^i\partial_v^jX\|_{a,t_0,i+2j} \le C\quad\text{with}\quad i+j = k,\\
& \label{induc1_3} \|\partial_x^i\partial_v^jV\|_{a,t_0,i+2j} \le C\quad\text{with}\quad i+j = k,
\end{align}
for all $2 \le k \le K$. Notice that the subindex $k$ appeared in the function space behaves like: each $x$-derivative counts for one, and each $v$-derivative counts for two.

We start with
\begin{equation}
\partial_y E(y,t) = \rho-\rho_0 = \int \delta(y-X) f^*(x,v) \rd{X}\rd{V} -\rho_0,
\end{equation} 
where the second equality is due to (\ref{varchange}) with $\phi$ being a delta function. Then taking $\partial_y^{k-1}$ gives
\begin{equation}\label{dykF}
\begin{split}
\partial_y^k E(y,t) = & \int \partial_y^{k-1}\delta(y-X) f^*(x,v) \rd{X}\rd{V} \\
= & \int \delta(y-X) \partial_X^{k-1}f^*(x,v) \rd{X}\rd{V} \\
= & \int \delta(y-X) (P^{f^*}_{k-1} + P^{f^*,1}_{k-1} + P^{f^*,0}_{k-1}) \rd{X}\rd{V} ,\\
\end{split}
\end{equation} 
by Lemma \ref{lem_dXx}. With $k=2$, the above equality takes the form
\begin{equation}
\partial_y^2 \opF = \int \delta(y-X) [\partial_x f^* + (\frac{\partial V}{\partial v}-1) \partial_x f^* - \frac{\partial V}{\partial x} \partial_v f^*]\rd{X}\rd{V}.
\end{equation}
Note that $(\frac{\partial V}{\partial v}-1)$ and $\frac{\partial V}{\partial x}$ are in $\opC_{a,t_0,2}$, and $\int \delta(y-X) [|\partial_x f^*| + |\partial_v f^*|] \rd{X}\rd{V}$ is bounded in $L^\infty$ (which is a consequence of ({\bf Decay}) of $\nabla_{x,v} f^*$, with the estimate (\ref{rhoest1})). The term $\int \delta(y-X) \partial_x f^* \rd{X}\rd{V}$ is in $\opC_{a,t_0,1}$ due to ({\bf Smoothness}) and ({\bf Decay}) of $\partial_x f^*$ with (\ref{rhoest2}). This finishes the proof of (\ref{induc1_1}) for $k=2$.

Then we prove (\ref{induc1_1}) for $k> 2$ based on the induction hypothesis. Due to the $L^\infty$ estimate for $\int \delta(y-X)|\partial_x^\alpha \partial_v^\beta f^*|\rd{X}\rd{V}$ and the $\opC_{a,t_0,1}$ estimate for $\int \delta(y-X)\partial_x^\alpha \partial_v^\beta f^*\rd{X}\rd{V}$ (see the previous paragraph), it suffices to prove that in $(P^{f^*}_{k-1} + P^{f^*,1}_{k-1} + P^{f^*,0}_{k-1})$, each coefficient in front of $\partial_x^\alpha \partial_v^\beta f^*$ is in $\opC_{a,t_0,k}$ or equal to $1$.

For the terms in $P^{f^*}_{k-1}$, since at least two factors are not $\frac{\partial V}{\partial v}$ (thus in $\opC_{a,t_0,k_1}$ for some $k_1$, by induction hypothesis and the fact that $\|\frac{\partial V}{\partial x}\|_{a,t_0,1}\le C$) and all other factors are in $L^\infty$, such term is in $\opC_{a,t_0,k}$ by Lemma \ref{lem_nl}. 

For the term $P^{f^*,1}_{k-1}$ given by
\begin{equation}
P^{f^*,1}_{k-1} = \left(\frac{\partial V}{\partial v}\right)^{k-2}\sum_{j=1}^{k-1}\partial_x^{k-1-j}\left[(j-1)\left(\frac{\partial^2 V}{\partial v\partial x}\partial_x^{j-1} f^*\right) - \left(\frac{\partial V}{\partial x}\partial_x^{j-1} \partial_v f^*\right)\right],
\end{equation}
the order of derivatives on $V$ inside the summation is at most $k-2$ times in $x$ and once in $v$. By induction hypothesis and the fact that $\|\frac{\partial V}{\partial x}\|_{a,t_0,1}\le C$ (where each $x$-derivative counts for one and $v$ for two on the function space subindex), the coefficients in $P^{f^*,1}_{k-1}$ are in $\opC_{a,t_0,k}$.

For the term 
\begin{equation}
P^{f^*,0}_{k-1}=\left(\frac{\partial V}{\partial v}\right)^{k-1}\partial_x^{k-1} f^*=\left[1 + \left(\frac{\partial V}{\partial v}-1\right)\sum_{j=0}^{k-2}\left(\frac{\partial V}{\partial v}\right)^j\right]\partial_x^{k-1} f^*,
\end{equation} 
its coefficient is a constant 1 plus terms in $\opC_{a,t_0,2}$. This finishes the proof of $\partial_y^k E \in \opC_{a,t_0,k}$.

Finally we prove (\ref{induc1_2}) and (\ref{induc1_3}) based on the induction hypothesis and (\ref{induc1_1}). Taking $\partial_x^i\partial_v^j$ (with $i+j=k\ge 2$) on (\ref{XV}) gives
\begin{equation}\label{dijX}
\begin{split}
& \partial_x^i\partial_v^jX =  \int_t^\infty (s-t) \left[\partial_xE(X(x,v,s),s)\partial_x^i\partial_v^j X(s) + \dots + \partial_x^k E \left(\frac{\partial X}{\partial x}\right)^i\left(\frac{\partial X}{\partial v}\right)^j \right]  \rd{s}, \\
& \partial_x^i\partial_v^jV = \int_t^\infty  \left[\partial_xE(X(x,v,s),s)\partial_x^i\partial_v^j X(s) + \dots + \partial_x^k E \left(\frac{\partial X}{\partial x}\right)^i\left(\frac{\partial X}{\partial v}\right)^j \right]   \rd{s},
\end{split}
\end{equation}
where all the omitted source terms consist of one derivative of $E$ (of order at most $k-1$) multiplied by some $x,v$-derivatives of $X$. In the first equation, all the source terms are in $\opC_{a,t_0,k+j}$ since: $\partial_x^i E$ is in $\opC_{a,t_0,i}$
 for $1\le i \le k$; all the derivatives of $X$ appeared are in $L^\infty$ except the term $\frac{\partial X}{\partial v}$ which is in $L^\infty_1$; $\frac{\partial X}{\partial v}$ has power at most $j$. Then one concludes that $ \partial_x^i\partial_v^jX$ is in 
$\opC_{a,t_0,i+2j}$ by Lemma \ref{lem_contra} with index $i+2j \le 2K$. Using this one obtains that $ \partial_x^i\partial_v^j V$ is in 
$\opC_{a,t_0,i+2j}$.

\section{Proof of Theorem \ref{thm}: estimates for $z$-derivatives}

Now we start estimating $\partial_z^k E$ where $E(z)$ is the fixed point of $\opF$. We use induction on $k$ to prove
\begin{align}
& \label{induc2_1} \|\partial_x^i\partial_v^j\partial_z^k X\|_{a,t_0,i+2j} \le C,\quad\text{for all}\quad i,j\ge 0,\quad i+j+k \le K, \quad\text{except}\quad k=0,\,i+j\le 1,\\
& \label{induc2_2} \|\partial_x^l\partial_z^k E\|_{a,t_0,l} \le C,\quad\text{for all}\quad l\ge 0, \quad l+k \le K,
\end{align}
for $k=0,\dots,K$, where the constants $C$ may depend on $a,a_1,a_2,t_0$ and the derivative indices. In particular, this will imply Theorem \ref{thm}.

The case $k=0$ is already proved, so we will assume $k\ge 1$ and prove (\ref{induc2_1}) and (\ref{induc2_2}) based on the induction hypothesis. For a fixed $k$, we will use induction on $l=i+j$. We first prove the case $l=0$, then prove the case $l\ge 1$ based on the induction hypothesis (on $l$).

\subsection{Case $l=0$: estimate for $\partial_z^k X$ and $\partial_z^k E$}

We start by estimating $\partial_z^k X$. Taking $\partial_z^k$ on (\ref{XV}),
\begin{equation}\label{dzkX}
\partial_z^k X(t) = \int_t^\infty (s-t)[\partial_x E(X(s),s)\partial_z^k X(s) + \partial_z^k E + P^E_k)]\rd{s},
\end{equation}
by Lemma \ref{lem_dzE}. By the induction hypothesis, each monomial in $P^E_k$ is a product of factors bounded in $L^\infty$, with at least one factor ($z$-derivative of $X$ of order between 1 and $k-1$) bounded in $\opC_{a,t_0}$. Thus $\|P^E_k\|_{a,t_0} \le C$. By Lemma \ref{lem_contra} with '$k=0$', we get
\begin{equation}\label{estdzkX}
\|\partial_z^k X\|_{a,t_0}  \le  \frac{4}{a^2-80a_2} \|\partial_z^k E\|_{a,t_0} + C,
\end{equation}
where $C$ comes from $P^E_k$. 

Then we estimate $\partial_z^k E$. Taking $\partial_z^k$ on (\ref{F}), using Lemma \ref{lem_dzB},
\begin{equation}\label{dzkF}
\begin{split}
\partial_z^k E(y,t) = & \int [-B'(y-X(x,v,t,z))\partial_z^kX(x,v,t,z) f^*(x,v) + P^B_k(B,f^*,X)] \rd{x}\rd{v} \\
= & \frac{1}{2\pi}\int \partial_z^k Xf^*(x,v) \rd{x}\rd{v} - \int \delta(y-X)\partial_z^k Xf^*(x,v) \rd{x}\rd{v}+ \int P^B_k\rd{x}\rd{v}\\
= & I_1 + I_2 + S.\\
\end{split}
\end{equation}
We estimate $I_1$ by
\begin{equation}
|I_1| \le \frac{1}{2\pi}\int f^*\rd{x}\rd{v} \|\partial_z^k X\|_{a,t_0}e^{-at} \le \frac{8a_2}{a^2-80a_2} \|\partial_z^k E\|_{a,t_0}e^{-at} + Ce^{-at},
\end{equation}
by (\ref{estdzkX}) and ({\bf Decay}) of $f^*$, and estimate $I_2$ by
\begin{equation}
\begin{split}
|I_2| \le & \|\partial_z^k X\|_{a,t_0}e^{-at}\int \delta(y-X)f(X,V,t) \rd{X}\rd{V} \le \left(\frac{4}{a^2-80a_2} \|\partial_z^k E\|_{a,t_0} + C\right)e^{-at}\rho(y,t) \\  \le & \frac{40a_2}{a^2-80a_2} \|\partial_z^k E\|_{a,t_0}e^{-at} + Ce^{-at},
\end{split}
 \end{equation}
by (\ref{estdzkX}) and (\ref{rhoest1}) applied to $f$. These estimates imply that
\begin{equation}\label{dzEest}
\|\partial_z^k E-S\|_{a,t_0} \le \frac{48a_2}{a^2-80a_2} \|\partial_z^k E\|_{a,t_0} + C \le \frac{1}{2} \|\partial_z^k E\|_{a,t_0} + C.
\end{equation}

Now we estimate $S$ in $\opC_{a,t_0}$. We will do this for each monomial appeared in $P^B_k$. By Lemma \ref{lem_dzB}, such a monomial has the form $\partial_z^\beta f^*\partial_x^\alpha B \prod_{j=1}^\alpha \partial_z^{\gamma_j+1}X$. Thus the corresponding integral takes the form (where we suppress the dependence on $t$)
\begin{equation}
\begin{split}
 & \int (\partial_x^\alpha B)(y-X) \prod_{j=1}^\alpha \partial_z^{\gamma_j+1}X \partial_z^\beta f^*(x,v)\rd{x}\rd{v} \\
= & \int (\partial_x^\alpha B)(y-X) \prod_{j=1}^\alpha (\partial_z^{\gamma_j+1}X)(x(X,V),v(X,V)) \partial_z^\beta f^*(x(X,V),v(X,V))\rd{X}\rd{V} \\
= & \int B(y-X) \frac{\partial^\alpha}{\partial X^\alpha}\left[\prod_{j=1}^\alpha (\partial_z^{\gamma_j+1}X)(x(X,V),v(X,V)) \partial_z^\beta f^*(x(X,V),v(X,V))\right]\rd{X}\rd{V}. \\
\end{split}
\end{equation}
This expression can be further expanded by using the product rule and Lemma \ref{lem_dXx}. 

If $\alpha\ge 2$, there are at least two factors being the derivatives of $X$ (with the order in $z$ between 1 and $k-1$, thus in $\opC_{a,t_0,k_1}$ for some $k_1$, after taking $X$-derivatives), while all other terms are in $L^\infty$. Thus the term is in $\opC_{a,t_0}$ by Lemma \ref{lem_nl}, since $\|B\|_{L^\infty}$ and $\|\partial_x^\gamma\partial_v^\delta\partial_z^\beta f^*(x,v)\|_{L^1}$ are finite. 

Thus the only nontrivial terms are those with $\alpha=1$. Such term has the form
\begin{equation}\label{alpha1}
\int (\delta(y-X)-\frac{1}{2\pi}) \partial_z^{\gamma+1} X \partial_z^\beta f^*(x,v)\rd{X}\rd{V},
\end{equation}
with $\gamma+\beta = k-1,\,\beta\ge 1$. By induction hypothesis $\partial_z^{\gamma+1}X\in \opC_{a,t_0}$. (\ref{rhoest1}) applied to $|\partial_z^\beta f^*|$ gives $\int (\delta(y-X)-\frac{1}{2\pi}) |\partial_z^\beta f^*(x,v)|\rd{X}\rd{V}\in L^\infty$, and thus (\ref{alpha1}) is in $\opC_{a,t_0}$. This finishes the induction for the case $l=0$, in view of (\ref{dzEest}) and (\ref{estdzkX}).
\subsection{Case $l\ge 1$: estimate for $x,v,z$-derivatives}
Taking $\partial_z^k$ on the first equation of (\ref{dijX}),
\begin{equation}\label{dX}
\begin{split}
\partial_x^i\partial_v^j\partial_z^k X(t) = &  \int_t^\infty (s-t)[ \partial_x E(X(s),s)\partial_x^i\partial_v^j\partial_z^k X(s) \\ & + \partial_x^l \partial_z^k E(X(s),s)\left(\frac{\partial X}{\partial x}\right)^i\left(\frac{\partial X}{\partial v}\right)^j + \dots]\rd{s}, \\
\end{split}
\end{equation}
where each factor in the omitted terms has $z$-derivatives of order at most $k-1$, or $x,v$-derivatives of total order at most $l-1$, thus all these factors can be controlled  by the induction hypothesis and the estimates for the $x,v$-derivatives of $X,V$ (see Section 4). We already know that all the terms in the bracket in the first equation of (\ref{dijX}) are in $\opC_{a,t_0,i+2j}$. According to the induction hypothesis, taking $z$-derivatives on the $x$-derivatives of $E$ and $x,v$-derivatives of $X$ does not make its decay property worse, except in the case when the $z$-derivative hits the $X$ inside $E$. In this exceptional case, there is an extra factor $\partial_z X\in \opC_{a,t_0}$ coming out. Then in this term there are two factors in $\opC_{a,t_0,k_1}$ for some $k_1$, namely, $\partial_z X$ and a derivative of $E$. Thus this term is in $\opC_{a,t_0}$ in view of Lemma \ref{lem_nl}. Therefore all the omitted terms in (\ref{dX}) are also in $\opC_{a,t_0,i+2j}$.
Thus Lemma \ref{lem_contra} with parameter '$k=l+j$' gives the estimates
\begin{equation}\label{dijkX}
\|\partial_x^i\partial_v^j\partial_z^k X\|_{a,t_0,l+j} \le 8\left(\left\|\frac{\partial X}{\partial x}\right\|_{L^\infty}^i\left\|\frac{\partial X}{\partial v}\right\|_{L^\infty_1}^j\|\partial_x^l\partial_z^k E\|_{a,t_0,l} + C_1\right),
\end{equation}
where the constant $C_1$ comes from the omitted terms. 

Then taking $\partial_y^l$ on (\ref{dzkF}), we have
\begin{equation}
\partial_y^l \partial_z^k E(y,t) = \partial_y^l I_1 + \partial_y^l I_2 + \partial_y^l S.
\end{equation}
It is clear that $\partial_y^l I_1 = 0$, and the main task is to estimate $\partial_y^l I_2$ and $\partial_y^l S$. We will show that $\|\partial_y^l I_2\|_{a,t_0,l} \le \frac{1}{2}\|\partial_x^l \partial_z^k E\|_{a,t_0,l} + C$, and $\|\partial_y^l S\|_{a,t_0,l} \le C$, which implies (\ref{induc2_2}) by contraction argument, and then (\ref{induc2_1}) follows from (\ref{dijkX}).
\begin{equation}\label{final0}
\begin{split}
\partial_y^l I_2 = &  \int \delta^{(l)}(y-X)\partial_z^k Xf^*(x,v) \rd{x}\rd{v} \\
 = &  \int \delta^{(l)}(y-X)(\partial_z^k X)(x,v)f^*(x,v) \rd{X}\rd{V}\\
= &  \int \delta(y-X)\partial_X^l[(\partial_z^k X)(x,v)f^*(x,v)] \rd{X}\rd{V}\\
= &  \int \delta(y-X)[\partial_X^l(\partial_z^k X)f^* + S_1] \rd{X}\rd{V}\\
= &  \int \delta(y-X)\left[\sum_{i+j=l}{l \choose i}\left(\frac{\partial V}{\partial v}\right)^i\left(-\frac{\partial V}{\partial x}\right)^j(\partial_x^i \partial_v^j\partial_z^k X)f^* + S_1+S_2\right] \rd{X}\rd{V}.\\
\end{split}
\end{equation}
Here the term $S_1$ contains all terms where there is at least one $X$-derivative hitting $f^*$, and $S_2$ contains all terms where at least one derivative in $\partial_X^l=(\frac{\partial V}{\partial v}\frac{\partial}{\partial x}-\frac{\partial V}{\partial x}\frac{\partial}{\partial v})^l$ hitting $\frac{\partial V}{\partial v}$ or $\frac{\partial V}{\partial x} $ in itself. Later $S_1$ and $S_2$ will be estimated by the induction hypothesis. The self-interacting term is estimated by 
\begin{equation}\label{final}
\begin{split}
& \left\|\sum_{i+j=l}{l \choose i}\left(\frac{\partial V}{\partial v}\right)^i\left(-\frac{\partial V}{\partial x}\right)^j(\partial_x^i \partial_v^j\partial_z^k X)\right\|_{a,t_0,l} \\ 
\le & \sum_{i+j=l}{l \choose i}[t_0^2e^{-at_0}]^j\left\|\frac{\partial V}{\partial v}\right\|_{L^\infty}^i \left\|\frac{\partial V}{\partial x}\right\|_{a,t_0,1}^j\|\partial_x^i \partial_v^j\partial_z^k X\|_{a,t_0,l+j} \\
\le & 8\sum_{i+j=l}{l \choose i}[t_0^2e^{-at_0}]^j\left\|\frac{\partial V}{\partial v}\right\|_{L^\infty}^i \left\|\frac{\partial V}{\partial x}\right\|_{a,t_0,1}^j \left\|\frac{\partial X}{\partial x}\right\|_{L^\infty}^i\left\|\frac{\partial X}{\partial v}\right\|_{L^\infty_1}^j\|\partial_x^l\partial_z^k E\|_{a,t_0,l} + C \\
\le & 8\left[ \left\|\frac{\partial V}{\partial v}\right\|_{L^\infty}\left\|\frac{\partial X}{\partial x}\right\|_{L^\infty} +    t_0^2e^{-at_0} \left\|\frac{\partial V}{\partial x} \right\|_{a,t_0,1}\left\|\frac{\partial X}{\partial v}\right\|_{L^\infty_1}\right]^l\|\partial_x^l\partial_z^k E\|_{a,t_0,l} + C \\
:= & 8 A^l \|\partial_x^l\partial_z^k E\|_{a,t_0,l} + C.
\end{split}
\end{equation}
Here the constant $[t_0^2e^{-at_0}]^j $ appeared in second line comes from the embedding $(\opC_{a,t_0,1})^j\times\opC_{a,t_0,l+j}\rightarrow \opC_{a,t_0,l}$ (see Lemma \ref{lem_nl}, where the condition $t_0 \ge t_1 = \frac{l+2j-l}{ja} = \frac{2}{a}$ is satisfied due to ({\bf A1}), ({\bf A2})).

Thus in order to use a contraction argument on $\partial_x^l\partial_z^k E$, one needs to require that $8A^l$ is no more than $\frac{1}{20a_2}$, in view of the fact that $\int \delta(y-X)|f^*|\rd{X}\rd{V}$ has $L^\infty$ norm at most $10a_2$ by (\ref{rhoest1}).

Now notice that all the $L^\infty$ norms in $A$ are with $t\ge t_0$. Therefore, from the $\opC_{a,t_0,1}$ and $\opC$ estimates on $X,V$ we obtained in section \ref{sec_XV}, we get
\begin{equation}
\begin{split}
& \left\|\frac{\partial X}{\partial x}\right\|_{L^\infty} \le 1 +  \frac{8}{a^2}C_Et_0e^{-at_0} , \\
& \left\|\frac{\partial V}{\partial v}\right\|_{L^\infty} \le 1 + \left\|\frac{\partial V}{\partial v} -1\right\|_{L^\infty} \le 1 + \left\|\frac{\partial V}{\partial v} -1\right\|_{a,t_0,2} t_0^2e^{-at_0} \le 1 + \frac{10C_E}{a} t_0^2e^{-at_0}  ,\\
& \left\|\frac{\partial V}{\partial x}\right\|_{a,t_0,1} \le \frac{4C_E}{a}, \\
&\left \|\frac{\partial X}{\partial v}\right\|_{L^\infty_1} \le 2 . \\
\end{split}
\end{equation}

Thus the bracket term in (\ref{final}) is at most $ 1 + \frac{50C_E}{a} t_0^2e^{-at_0}$, in view of the fact that $\frac{8C_E}{a^2} t_0e^{-at_0} \le 1$ by ({\bf A5}).
Notice that
\begin{equation}
\left(1 + \delta \right)^l \le e^{l\delta}.
\end{equation}
With $\delta = \frac{50C_E}{a} t_0^2e^{-at_0}$, 
\begin{equation}
l \delta = \frac{50C_E}{a} lt_0^2e^{-at_0} \le \frac{50C_E}{a} t_0^3e^{-at_0} \le 1,
\end{equation}
by ({\bf A3}) and $l\le K \le t_0$ (a consequence of ({\bf A2})). Thus $8A^l \le 8e \le \frac{1}{20a_2}$ by ({\bf A4}).

Next we estimate the terms $S_2$ and $S_1$. 

All terms in $S_2$ has $x,v$-derivatives in $X$ of total order at most $l-1$, thus can be controlled by the induction hypothesis. In each term, at least two factors are in $\opC_{a,t_0,k_1}$ for some $k_1$ (one is a derivative of $\partial_z^k X$, another is a $x,v$-derivative of $V$ of total order at least 2). This shows that $S_2$ is in $\opC_{a,t_0,l}$, in view of the fact that $\int \delta(y-X)| f^*|\rd{X}\rd{V}$ is in $L^\infty$.

$S_1 = \sum_{l'=0}^{l-1}{l \choose l'} \partial_X^{l'}(\partial_z^k X) \partial_X^{l-l'} f^*$. The term $\partial_X^{l'}(\partial_z^k X)$ has at least two factors in $\opC_{a,t_0,k_1}$ for some $k_1$, except the terms with all derivatives inside $\partial_X^{l'} =(\frac{\partial V}{\partial v}\frac{\partial}{\partial x}-\frac{\partial V}{\partial x}\frac{\partial}{\partial v})^{l'}$ being $\partial_x$ and hitting on $\partial_z^k X$. This term is $(\frac{\partial V}{\partial v})^{l'}\partial_x^{l'}\partial_z^k X$, which is in $\opC_{a,t_0,l'}$. This shows that $S_2$ is in $\opC_{a,t_0,l}$, in view of the fact that $\int \delta(y-X)| \partial_X^{l-l'}f^*|\rd{X}\rd{V}$ is in $L^\infty$ (which is clear since after expanding $\partial_X^{l-l'}$, all coefficients in front of $f^*$ are in $L^\infty$).

Finally, to estimate $\partial_y^l S$, we only need to treat $\partial_y^l$ of (\ref{alpha1}) for the same reason as before. This term is
\begin{equation}
\begin{split}
 \int \delta^{(l)}(y-X) \partial_z^{\gamma+1} X \partial_z^\beta f^*(x,v)\rd{X}\rd{V} 
=  \int \delta(y-X) \partial_X^l[\partial_z^{\gamma+1} X \partial_z^\beta f^*(x,v)]\rd{X}\rd{V}.
\end{split}
\end{equation}
This term is bounded in $\opC_{a,t_0,l}$ in the same way as $\partial_y^l I_2$ is bounded (with $k$ replaced by $\gamma+1 < k$, and $f^*$ replaced by $\partial_z^\beta f^*$), while here we use the induction hypothesis instead of self-interacting estimates, since the involved $z$-derivative is at most $k-1$.

\section{Proof of Corollary \ref{cor}}

We first estimate the derivatives of $V$. Take $i,j,k$ such that $k=0,\,i+j\le 1$ does not hold. Taking $\partial_z^k$ of the second equation of (\ref{dijX}), one gets (similar to (\ref{dX}))
\begin{equation}
\begin{split}
\partial_x^i\partial_v^j\partial_z^k V(t) =   -\int_t^\infty \left[ \partial_x E(X(s),s)\partial_x^i\partial_v^j\partial_z^k X(s) + \partial_x^l \partial_z^k E(X(s),s)\left(\frac{\partial X}{\partial x}\right)^i\left(\frac{\partial X}{\partial v}\right)^j + \dots\right]\rd{s}, \\
\end{split}
\end{equation}
where the omitted term is in $\opC_{a,t_0,i+2j}$ as in (\ref{dX}). By Section 5, $\partial_x^i\partial_v^j\partial_z^k X$ is in $\opC_{a,t_0,i+2j}$ and $\partial_x^l \partial_z^k E(X(s),s)$ is in $\opC_{a,t_0,i+j}$. It follows that the integrand is in $\opC_{a,t_0,i+2j}$, and one concludes that $\partial_x^i\partial_v^j\partial_z^k V \in \opC_{a,t_0,i+2j}$ by Lemma \ref{lem_op1}.

$f(x,v,t,z)$, the solution to (\ref{VP}) given by \cite{CM} with time-asymptotic profile $f^*(x,v,z)$, is given by the implicit form (\ref{deff}). Using (\ref{dkg}) and its variant for $V$-derivatives, one obtains
\begin{equation}\label{dxvf}
\partial_x^i\partial_v^j f = \left(\frac{\partial V}{\partial v}\partial_x  - \frac{\partial V}{\partial x}\partial_v\right)^i\left(-\frac{\partial X}{\partial v}\partial_x  + \frac{\partial X}{\partial x}\partial_v\right)^j f^*,
\end{equation}
where $f$ and its derivatives are evaluated at $(X,V,t,z)$, $(X,V)$ and their derivatives are evaluated at $(x,v,t,z)$, and $f^*$ and its derivatives are evaluated at $(x,v,z)$. 

Taking $\partial_z^k$ of the RHS of  (\ref{dxvf}), due to the fact that all derivatives of $X,V$ of order at least two are in $\opC_{a,t_0,k_1}$ for some $k_1$, the worst term is when all derivatives hit $f^*$, and the $X,V$ factors are chosen as $ \left(\frac{\partial V}{\partial v}\right)^i \left(\frac{\partial X}{\partial v}\right)^j$. This term is in $L^\infty_j$, and thus $\frac{\partial^k}{\partial z^k}[\partial_x^i\partial_v^j f(X(z),V(z),t,z)]$ is in $L^\infty_j$.

Taking $\partial_z^k$ of $(\partial_x^i\partial_v^j f)(X(z),V(z),t,z) $, using a variant of Lemma \ref{lem_dzE}, one obtains
\begin{equation}\label{df}
\frac{\partial^k}{\partial z^k}[\partial_x^i\partial_v^j f(X(z),V(z),t,z)] = \sum_{\alpha_1+\alpha_2+\beta+\gamma_1+\gamma_2= k} \left[ \partial_x^{i+\alpha_1}\partial_v^{j+\alpha_2}\partial_z^\beta f \sum c\prod_{l=1}^{\alpha_1} \partial_z^{\gamma_{1,l}+1}X\prod_{l=1}^{\alpha_2} \partial_z^{\gamma_{2,l}+1}V \right],
\end{equation}
where the second summation is taken over $\gamma_{1,1}+\dots+\gamma_{1,\alpha_1}=\gamma_1,\,\gamma_{2,1}+\dots+\gamma_{2,\alpha_2}=\gamma_2$. On the RHS there is one single term $\partial_x^i\partial_v^j\partial_z^k f$ with $\beta=k$, and all other terms have $\beta<k$. Thus one can write
\begin{equation}\label{df1}
\begin{split}
\partial_x^i\partial_v^j\partial_z^k f = &  \frac{\partial^k}{\partial z^k}[\partial_x^i\partial_v^j f(X(z),V(z),t,z)] \\ & - \sum_{\alpha_1+\alpha_2+\beta+\gamma_1+\gamma_2= k, \,\beta<k} \left[ \partial_x^{i+\alpha_1}\partial_v^{j+\alpha_2}\partial_z^\beta f \sum c\prod_{l=1}^{\alpha_1} \partial_z^{\gamma_{1,l}+1}X\prod_{l=1}^{\alpha_2} \partial_z^{\gamma_{2,l}+1}V \right].
\end{split}
\end{equation}
In case $\beta<k$, one has $\alpha_1+\alpha_2 \ge 1$, thus there is at least one factor which is a $z$-derivative of $X$ or $V$. Such factor is in $\opC_{a,t_0}$. Thus by induction on $k$, it is easy to prove that $\partial_x^i\partial_v^j\partial_z^k f$ is in $L^\infty_j$, using the fact that $\frac{\partial^k}{\partial z^k}[\partial_x^i\partial_v^j f(X(z),V(z),t,z)] $ is in this space.

Now let $i=j=0$ in (\ref{df1}). The first term on the RHS becomes $\partial_z^k f^*(x,v,z)$. For the terms in the summation, if $\alpha_1+\alpha_2\ge 2$, then there are two factors in $\opC_{a,t_0}$, while others are no worse than $L^\infty_{k_1}$ for some $k_1$, and thus the term is in $\opC_{a,t_0}$ in view of Lemma \ref{lem_nl}. If $\alpha_1+\alpha_2=1$, then the $f$ factor is $\partial_x\partial_z^\beta f$ or $\partial_v\partial_z^\beta f$, both of which are in $L^\infty_1$. Together with one $X$ or $V$ factor in $\opC_{a,t_0}$, the term is in $\opC_{a,t_0,1}$. Thus the whole summation is in $\opC_{a,t_0,1}$, and we obtain 
\begin{equation}\label{ff1}
\partial_z^k f(X,V,t,z) - \partial_z^k f^*(x,v,z) \in \opC_{a,t_0,1}.
\end{equation}

Using the fact that $|v-V|\le Ce^{-at}$, $|x-(X-Vt)| \le Cte^{-at}$, and the assumption that $\nabla_{x,v}\partial_z^k f^*$ is bounded in $L^\infty$, we obtain
\begin{equation}
\partial_z^k f^*(x,v,z) - \partial_z^k f^*(X-Vt,V,z) \in \opC_{a,t_0,1},
\end{equation}
which implies
\begin{equation}
\partial_z^k f(X,V,t,z)- \partial_z^k f^*(X-Vt,V,z) \in \opC_{a,t_0,1},
\end{equation}
and the conclusion of Corollary \ref{cor} follows.

\section{Conclusion}
In this paper we proved that for the Vlasov-Poisson equation with random uncertain initial data, the Landau damping solution $E(t,x,z)$  given by \cite{CM} (for the deterministic problem) depends smoothly on the random variable $z$, if the time asymptotic profile $f^*(x,v,z)$ does. Our smoothness and smallness assumptions on $f^*$ are similar to those in \cite{CM}, and independent of $K$, the order of $z$-derivatives. 

To the authors knowledge, this result is the first mathematical study on the propagation of uncertainty for time-reversible nonlinear kinetic equations. It suggests that even for kinetic equations without hypocoercivity, the random space regularity may still be maintained in large time, if there are other types of damping mechanism (Landau damping for the VP equation).

In the future we may consider:
\begin{enumerate}
\item How to extend the results in~\cite{MV} and~\cite{BMM} to the case with uncertainty.
\item There are other equations for which the phase-mixing mechanism induces damping, for example, the 2D Euler equation \cite{BM}. It is interesting to see whether such damping result can be extended to the case with uncertainty. 
\end{enumerate}

\bibliographystyle{plain}

\end{document}